\documentclass[12pt]{amsart}

\usepackage[matrix,arrow,curve,frame]{xy}
\usepackage{amsmath,amsthm,amssymb,enumerate}
\usepackage{latexsym}
\usepackage{amscd}
\usepackage[colorlinks=false]{hyperref}
\usepackage{euscript}

\newtheorem{thm}[subsection]{Theorem}

\newtheorem{lemma}[subsection]{Lemma}

\newtheorem{remark}[subsection]{Remark}

\theoremstyle{definition}
\newtheorem{example}[subsection]{Example}

\numberwithin{equation}{section}

\def\cO{{\cal O}}

\def\cA{{\cal A}}

\def\ra{\rightarrow}
\def\bra{\langle}
\def\ket{\rangle}

\def\cA{{\mathcal A}}
\def\cB{{\mathcal B}}
\def\cC{{\mathcal C}}

\def\cE{{\mathcal E}}
\def\cF{{\mathcal F}}

\def\cH{{\mathcal H}}

\def\cO{{\mathcal O}}

\def\cS{{\mathcal S}}

\def\cV{{\mathcal V}}
\def\cW{{\mathcal W}}

\def\gg{{\mathfrak g}}

\def\gl{{\mathfrak l}}

\def\go{{\mathfrak o}}
\def\gp{{\mathfrak p}}

\def\gs{{\mathfrak s}}

\newfont{\german}{eufm10}

\begin{document}
\pagestyle{plain}

\title
{A commutant realization of Odake's algebra}

\author{Thomas Creutzig and Andrew R. Linshaw}

\address{Department of Mathematics, University of Alberta}
\email{creutzig@ualberta.ca}

\address{Department of Mathematics, University of Denver}
\email{andrew.linshaw@du.edu}
\thanks{We thank R. Heluani for interesting discussions and for suggesting that Odake's algebra should appear inside the vertex algebra $\text{Com}(V_0(\gs\gl_2), \cW)$.}


{\abstract

\noindent
The $bc\beta\gamma$-system $\cW$ of rank $3$ has an action of the affine vertex algebra $V_0(\gs\gl_2)$, and the commutant vertex algebra $\cC = \text{Com}(V_0(\gs\gl_2), \cW)$ contains copies of $V_{-3/2}(\gs\gl_2)$ and Odake's algebra $\cO$. Odake's algebra is an extension of the $N=2$ superconformal algebra with $c=9$, and is generated by eight fields which close nonlinearly under operator product expansions. Our main result is that $V_{-3/2}(\gs\gl_2)$ and $\cO$ form a Howe pair (i.e., a pair of mutual commutants) inside $\cC$. More generally, any finite-dimensional representation of a Lie algebra $\gg$ gives rise to a similar Howe pair, and this example corresponds to the adjoint representation of $\gs\gl_2$.}

\maketitle


\section{Introduction}

\noindent

Let $\cV$ be a vertex algebra, and let $\cA$ be a subalgebra of $\cV$. The {\it commutant} of $\cA$ in $\cV$, denoted by $\text{Com}(\cA,\cV)$, is the subalgebra consisting of all elements $v\in\cV$ such that $ [a(z),v(w)] = 0$ for all $a\in \cA$. This construction was introduced by Frenkel-Zhu in \cite{FZ}, generalizing earlier constructions in representation theory \cite{KP} and physics \cite{GKO}, and is important in the construction of coset conformal field theories. Many interesting vertex algebras have commutant realizations. For example, the Zamolodchikov $\cW_3$-algebra with central charge $c=-2$ can be realized as the commutant of the Heisenberg algebra inside the $\beta\gamma$-system \cite{Wa}. The Feigin-Semikhatov algebra $\cW^{(2)}_n$ at critical level can be realized as the commutant of $V_0(\gp\gs\gl(n|n))$ inside the $bc\beta\gamma$-system of rank $n^2$ for $n\leq 4$, and conjecturally for all $n$ \cite{CGL}. The orbifold vertex algebra $M(1)^+$ can be realized as the commutant of  the affine vertex algebra $L_{C^{(1)}_l} (- \Lambda_0)$ inside the tensor product of two copies of $L_{C^{(1)}_l} (-\frac{1}{2} \Lambda_0)$ \cite{AP}. Here $M(1)^+$ denotes the $\mathbb{Z}/2\mathbb{Z}$-invariant subalgebra of the rank $l$ Heisenberg algebra, which is an important building block of the orbifold vertex algebra $V_L^+$ for any lattice $L$ of rank $l$. In general, commutants are difficult to study and there are very few examples where an exhaustive description can be given in terms of generators, operator product expansions, and normally ordered polynomial relations among the generators. It is also natural to study the double commutant $\text{Com}(\text{Com}(\cA,\cV),\cV)$, which always contains $\cA$. If $\cA = \text{Com}(\text{Com}(\cA,\cV),\cV)$, we say that $\cA$ and $\text{Com}(\cA,\cV)$ form a {\it Howe pair} inside $\cV$.

We begin with a simple, finite-dimensional Lie algebra $\gg$ and a finite-dimensional linear representation $V$ of $\gg$, via $\rho:\gg\rightarrow \text{End}(V)$. Associated to $\rho$ is the bilinear form $B(\xi,\eta) = \text{tr}_V(\rho(\xi) \rho(\eta))$ on $\gg$. There is a natural action of the affine vertex algebra $V_1(\gg,B)$ on the $bc$-system $\cE  = \cE(V)$ associated to $V$, and an analogous action of $V_{-1}(\gg,B)$ on the $\beta\gamma$-system $\cS = \cS(V)$. These combine to give an action of the level zero affine vertex algebra $V_0(\gg)$ on the $bc\beta\gamma$-system $\cW = \cE\otimes \cS$. Let $\Theta_{\cE}$, $\Theta_{\cS}$, and $\Theta_{\cW}$ denote the images of $V_1(\gg,B)$, $V_{-1}(\gg,B)$, and $V_0(\gg)$ inside $\cE$, $\cS$, and $\cW$, respectively. The commutants $\text{Com}(\Theta_{\cE},\cE)$, $\text{Com}(\Theta_{\cS},\cS)$, and $\text{Com}(\Theta_{\cW},\cW)$, coincide with the invariant spaces $\cE^{\gg[t]}$, $\cS^{\gg[t]}$, and $\cW^{\gg[t]}$, respectively. Since $\cS^{\gg[t]}\subset\cW^{\gg[t]}$, we may define \begin{equation*} \cC(\gg,V) = \text{Com}(\cS^{\gg[t]}, \cW^{\gg[t]}).\end{equation*} Clearly $\cE^{\gg[t]} \subset \cC(\gg,V)$, but in general $\cC(\gg,V)$ is larger than $\cE^{\gg[t]}$ and contains fields that depend on the generators $\beta^x, \gamma^{x'}$ of $\cS$. We establish some basic properties of $\cC(\gg,V)$, and in particular we show that $\cS^{\gg[t]}$ and $\cC(\gg,V)$ always form a Howe pair inside $\cW^{\gg[t]}$.

In the case where $\gg = \gs\gl_2$ and $V$ is the adjoint module $\mathbb{C}^3$, $\cS^{\gs\gl_2[t]}$ and $\cE^{\gs\gl_2[t]}$ are isomorphic to $V_{-3/2}(\gs\gl_2)$ and a rank-one lattice vertex algebra, respectively. Our main result is that $\cC(\gs\gl_2,\mathbb{C}^3)$ is isomorphic to Odake's algebra $\cO$. This algebra is an extension of the $N=2$ superconformal algebra with central charge $c=9$, and is strongly generated by a $U(1)$ current $F$, a Virasoro element $L$, four primary fields $G, \bar{G},X, \bar{X}$ of weight $\frac{3}{2}$, and two primary fields $Y, \bar{Y}$ of weight $2$. The algebra $\cO$ appeared first in the context of superstring theory. The target space of a ten-dimensional superstring is $\mathbb R^4\times X$, where $X$ is a three-dimensional Calabi-Yau manifold. The sigma model of such a Calabi-Yau manifold has $N=(2,2)$ superconformal symmetry, and due to Ricci-flatness an additional symmetry associated to the holomorphic $(3,0)$ and anti-holomorphic $(0,3)$ forms. The main result of \cite{OI} is that the underlying conformal field theory is two copies of $\cO$. On a similar note, Malikov-Schechtman-Vaintrob have associated to any smooth manifold a sheaf of vertex algebra known as the {\it chiral de Rham sheaf} \cite{MSV}, and it was recently shown in \cite{EHKZ} that the algebra of global sections $\Omega^{ch}(X)$ contains two commuting copies of $\cO$. 

It is interesting to study the vertex algebras $\cS^{\gg[t]}$, $\cE^{\gg[t]}$, $\cW^{\gg[t]}$, and $\cC(\gg,V)$ for any $\gg$ and $V$, and it appears that they have nice structures in general. In the last section, we identify these algebras in the case where $\gg = \gs\gl_2$ and $V=\mathbb{C}^2$. We will see that $\cS^{\gs\gl_2[t]}$, $\cE^{\gs\gl_2[t]}$, and $\cW^{\gs\gl_2[t]}$ are isomorphic to the Heisenberg algebra, the irreducible affine vertex algebra $L_1(\gs\gl_2)$, and a homomorphic image of $V_1(\gs\gl(2|1))$, respectively. Finally, $\cC(\gs\gl_2,\mathbb{C}^2)$ is isomorphic to $L_1(\gs\gl_2) \otimes \cW_{3,-2}$, where $\cW_{3,-2}$ is the Zamolodchikov $\cW_3$-algebra with $c=-2$.

\section{Vertex algebras}
In this section, we define vertex algebras, which have been discussed from various different points of view in the literature (see for example \cite{B}\cite{FLM}\cite{K}\cite{FBZ}). We will follow the formalism developed in \cite{LZ} and partly in \cite{LiI}. Let $V=V_0\oplus V_1$ be a super vector space over $\mathbb{C}$, and let $z,w$ be formal variables. By $\text{QO}(V)$, we mean the space of all linear maps $$V\rightarrow V((z)):=\{\sum_{n\in\mathbb{Z}} v(n) z^{-n-1}|
v(n)\in V,\ v(n)=0\ \text{for} \ n >\!\!> 0 \}.$$ Each element $a\in \text{QO}(V)$ can be
uniquely represented as a power series
$$a=a(z):=\sum_{n\in\mathbb{Z}}a(n)z^{-n-1}\in \text{End}(V)[[z,z^{-1}]].$$ We
refer to $a(n)$ as the $n$th Fourier mode of $a(z)$. Each $a\in
\text{QO}(V)$ is assumed to be of the shape $a=a_0+a_1$ where $a_i:V_j\ra V_{i+j}((z))$ for $i,j\in\mathbb{Z}/2\mathbb{Z}$, and we write $|a_i| = i$.

On $\text{QO}(V)$ there is a set of nonassociative bilinear operations
$\circ_n$, indexed by $n\in\mathbb{Z}$, which we call the $n$th circle
products. For homogeneous $a,b\in \text{QO}(V)$, they are defined by
$$
a(w)\circ_n b(w)=\text{Res}_z a(z)b(w)~\iota_{|z|>|w|}(z-w)^n-
(-1)^{|a||b|}\text{Res}_z b(w)a(z)~\iota_{|w|>|z|}(z-w)^n.
$$
Here $\iota_{|z|>|w|}f(z,w)\in\mathbb{C}[[z,z^{-1},w,w^{-1}]]$ denotes the
power series expansion of a rational function $f$ in the region
$|z|>|w|$. We usually omit the symbol $\iota_{|z|>|w|}$ and just
write $(z-w)^{-1}$ to mean the expansion in the region $|z|>|w|$,
and write $-(w-z)^{-1}$ to mean the expansion in $|w|>|z|$. It is
easy to check that $a(w)\circ_n b(w)$ above is a well-defined
element of $\text{QO}(V)$.

The nonnegative circle products are connected through the {\it
operator product expansion} formula.
For $a,b\in \text{QO}(V)$, we have \begin{equation} \label{opeformula} a(z)b(w)=\sum_{n\geq 0}a(w)\circ_n
b(w)~(z-w)^{-n-1}+:a(z)b(w):,\end{equation} which is often written as
$a(z)b(w)\sim\sum_{n\geq 0}a(w)\circ_n b(w)~(z-w)^{-n-1}$, where
$\sim$ means equal modulo the term $$
:a(z)b(w): \ =a(z)_-b(w)\ +\ (-1)^{|a||b|} b(w)a(z)_+.$$ Here
$a(z)_-=\sum_{n<0}a(n)z^{-n-1}$ and $a(z)_+=\sum_{n\geq
0}a(n)z^{-n-1}$. Note that $:a(w)b(w):$ is a well-defined element of
$\text{QO}(V)$. It is called the {\it Wick product} of $a$ and $b$, and it
coincides with $a\circ_{-1}b$. The other negative circle products
are related to this by
$$ n!~a(z)\circ_{-n-1}b(z)=\ :(\partial^n a(z))b(z):\ ,$$
where $\partial$ denotes the formal differentiation operator
$\frac{d}{dz}$. For $a_1(z),\dots ,a_k(z)\in \text{QO}(V)$, the $k$-fold
iterated Wick product is defined to be
$$ :a_1(z)a_2(z)\cdots a_k(z):\ =\ :a_1(z)b(z):,$$
where $b(z)=\ :a_2(z)\cdots a_k(z):$. We often omit the formal variable $z$ when no confusion can arise.

The set $\text{QO}(V)$ is a nonassociative algebra with the operations
$\circ_n$ and a unit $1$. We have $1\circ_n a=\delta_{n,-1}a$ for
all $n$, and $a\circ_n 1=\delta_{n,-1}a$ for $n\geq -1$. A linear subspace $\cA\subset \text{QO}(V)$ containing 1 which is closed under the circle products will be called a {\it quantum operator algebra} (QOA).
In particular $\cA$ is closed under $\partial$
since $\partial a=a\circ_{-2}1$. Many formal algebraic
notions are immediately clear: a homomorphism is just a linear
map that sends $1$ to $1$ and preserves all circle products; a module over $\cA$ is a
vector space $M$ equipped with a homomorphism $\cA\rightarrow
\text{QO}(M)$, etc. A subset $S=\{a_i|\ i\in I\}$ of $\cA$ is said to generate $\cA$ if any element $a\in\cA$ can be written as a linear
combination of nonassociative words in the letters $a_i$, $\circ_n$, for
$i\in I$ and $n\in\mathbb{Z}$. We say that $S$ {\it strongly generates} $\cA$ if any $a\in\cA$ can be written as a linear combination of words in the letters $a_i$, $\circ_n$ for $n<0$. Equivalently, $\cA$ is spanned by the collection $\{ :\partial^{k_1} a_{i_1}(z)\cdots \partial^{k_m} a_{i_m}(z):| ~i_1,\dots,i_m \in I,~ k_1,\dots,k_m \geq 0\}$.

We say that $a,b\in \text{QO}(V)$ {\it quantum commute} if $(z-w)^N
[a(z),b(w)]=0$ for some $N\geq 0$. Here $[,]$ denotes the super bracket. This condition implies that $a\circ_n b = 0$ for $n\geq N$, so (\ref{opeformula}) becomes a finite sum. A {\it commutative quantum operator algebra} (CQOA) is a QOA whose elements pairwise quantum commute. Finally, the notion of a CQOA is equivalent to the notion of a vertex algebra. Every CQOA $\cA$ is itself a faithful $\cA$-module, called the {\it left regular
module}. Define
$$\rho:\cA\rightarrow \text{QO}(\cA),\ \ \ \ a\mapsto\hat a,\ \ \ \ \hat
a(\zeta)b=\sum_{n\in\mathbb{Z}} (a\circ_n b)~\zeta^{-n-1}.$$ Then $\rho$ is an injective QOA homomorphism,
and the quadruple of structures $(\cA,\rho,1,\partial)$ is a vertex
algebra in the sense of \cite{FLM}. Conversely, if $(V,Y,{\bf 1},D)$ is
a vertex algebra, the collection $Y(V)\subset \text{QO}(V)$ is a
CQOA. {\it We will refer to a CQOA simply as a
vertex algebra throughout the rest of this paper}.

\begin{example}[Affine vertex algebras] Let $\gg$ be a finite-dimensional, complex Lie (super)algebra, equipped with a symmetric, invariant bilinear form $B$. The loop algebra $\gg[t,t^{-1}] = \gg\otimes \mathbb{C}[t,t^{-1}]$ has a one-dimensional central extension $\hat{\gg} = \gg[t,t^{-1}]\oplus \mathbb{C}\kappa$ determined by $B$, with bracket $$[\xi t^n, \eta t^m] = [\xi,\eta] t^{n+m} + n B(\xi,\eta) \delta_{n+m,0} \kappa,$$ and $\mathbb{Z}$-gradation $\text{deg}(\xi t^n) = n$, $\text{deg}(\kappa) = 0$. Let $\hat{\gg}_{\geq 0} = \bigoplus_{n\geq 0} \hat{\gg}_n$ where $\hat{\gg}_n$ denotes the subspace of degree $n$, and let $C$ be the one-dimensional $\hat{\gg}_{\geq 0}$-module on which $\xi t^n$ acts trivially for $n\geq 0$, and $\kappa$ acts by $k$ times the identity. Define $V = U(\hat{\gg})\otimes_{U(\hat{\gg}_{\geq 0})} C$, and let $X^{\xi}(n)\in \text{End}(V)$ be the linear operator representing $\xi t^n$ on $V$. Define $X^{\xi} (z) = \sum_{n\in\mathbb{Z}} X^{\xi} (n) z^{-n-1}$, which is easily seen to lie in $\text{QO}(V)$ and satisfy the operator product relation $$X^{\xi}(z)X^{\eta} (w)\sim kB(\xi,\eta) (z-w)^{-2} + X^{[\xi,\eta]}(w) (z-w)^{-1} .$$ The vertex algebra $V_k(\gg,B)$ generated by $\{X^{\xi}| \ \xi \in\gg\}$ is known as the {\it universal affine vertex algebra} associated to $\gg$ and $B$ at level $k$. If $\gg$ is a simple, finite-dimensional Lie algebra, we will always take $B$ to be the normalized Killing form $\frac{1}{2 h^{\vee}} \bra, \ket_K$, and we use the notation $V_k(\gg)$. 

We recall the \emph{Sugawara construction} for affine vertex (super)algebras following \cite{KRW}.
Suppose that $\gg$ is simple and that the bilinear form $B$ is nondegenerate. Let $\{\xi\}$ and $\{\xi'\}$ be dual bases of $\gg$, i.e., $B(\xi',\eta)=\delta_{\xi,\eta}$.
Then the Casimir operator is $C_2=\sum_{\xi}\xi\xi'$.
The dual Coxeter number $h^\vee$ with respect to the bilinear form $B$ is one-half the eigenvalue of $C_2$ in the adjoint representation of $\gg$.
If $k+h^\vee\neq0$, there is a Virasoro field
\begin{equation*}
L_{\text{Sug}} = \frac{1}{2(k+h^\vee)}\sum_\xi :X^{\xi}X^{\xi'}:
\end{equation*}
of central charge $c= \frac{k\text{sdim}\gg}{k+h^\vee}$, which is known as the {\it Sugawara conformal vector}. 
\end{example}
 
\begin{example}[$\beta\gamma$ and $bc$ systems] Let $V$ be a finite-dimensional complex vector space. The $\beta\gamma$ system or algebra of chiral differential operators $\cS = \cS(V)$ was introduced in \cite{FMS}. It is the unique vertex algebra with even generators $\beta^{x}$, $\gamma^{x'}$ for $x\in V$, $x'\in V^*$, which satisfy
\begin{equation*}
\begin{split}
\beta^x(z)\gamma^{x'}(w)&\sim\langle x',x\rangle (z-w)^{-1},\ \ \ \ \ \ \gamma^{x'}(z)\beta^x(w)\sim -\langle x',x\rangle (z-w)^{-1},\\
\beta^x(z)\beta^y(w)&\sim 0,\qquad\qquad\qquad\ \ \ \ \ \gamma^{x'}(z)\gamma^{y'}(w)\sim 0.
\end{split}
\end{equation*} Here $\bra,\ket$ denotes the natural pairing between $V^*$ and $V$. We give $\cS$ the conformal structure \begin{equation} \label{virasorobg} L_{\cS} = \sum_{i=1}^n :\beta^{x_i}\partial\gamma^{x'_i}:,\end{equation} of central charge $c=2n$, under which $\beta^{x_i}$, $\gamma^{x'_i}$ are primary of weights $1$, $0$, respectively. Here $\{x_1,\dots,x_n\}$ is a basis for $V$ and $\{x'_1,\dots,x'_n\}$ is the dual basis for $V^*$. 

Similarly, the $bc$ system $\cE = \cE(V)$, which was also introduced in \cite{FMS}, is the unique vertex superalgebra with odd generators $b^{x}$, $c^{x'}$ for $x\in V$, $x'\in V^*$, which satisfy
\begin{equation*}
\begin{split}
b^x(z)c^{x'}(w)&\sim\langle x',x\rangle (z-w)^{-1},\ \ \ \ \ \ c^{x'}(z)b^x(w)\sim \langle x',x\rangle (z-w)^{-1},\\
b^x(z)b^y(w)&\sim 0,\qquad\qquad\qquad\ \ \ \ \ c^{x'}(z)c^{y'}(w)\sim 0.
\end{split}
\end{equation*}
 We give $\cE$ the conformal structure \begin{equation*} \label{virasorobc} L_{\cE} = -\sum_{i=1}^n :b^{x_i}\partial c^{x'_i}:,\end{equation*} of central charge $c=-2n$, under which $b^{x_i}$, $c^{x'_i}$ are primary of conformal weights $1$, $0$, respectively. 

Let $\cW = \cE\otimes \cS$, equipped with the combined Virasoro element 
\begin{equation*} L_{\cW} = L_{\cE} + L_{\cS}\end{equation*} 
of central charge $c=0$.
\end{example}

\begin{example}[Odake's algebra] 
Consider a rank 6 Heisenberg vertex algebra with generators $\alpha^\pm_i$, $i=1,2,3$ and non-regular operator products 
$$ \alpha^+_i(z)\alpha^-_i(w)\sim (z-w)^{-2},$$
tensored with a rank three $bc$-system with 
generators $b_i,c_i$, $i=1,2,3$ and non-regular operator products 
$$ b_i(z)c_i(w)\sim (z-w)^{-1}.$$
In \cite{OI}, Odake defines a vertex superalgebra $\cO$ 
as the subalgebra generated by the fields $G,\bar G, X, \bar X$ given by
\begin{equation*}
\begin{split}
G&= \sum_{i=1}^3:b_i\alpha^+_i:,\qquad \bar G= \sum_{i=1}^3:c_i\alpha^-_i:,\qquad X=\ :b_1b_2b_3:,\qquad \bar X=\ :c_1c_2c_3:.
\end{split}
\end{equation*}
Define additional fields \begin{equation*} \label{additionalfields}Y = \frac{1}{2} \bar{G} \circ_1 X,\qquad \bar{Y} = \frac{1}{2} G\circ_1 \bar{X},\qquad F = G\circ_2 \bar{G},\qquad L = G\circ_1 \bar{G} - \frac{1}{2} \partial F.\end{equation*}
It turns out that $\cO$ is {\it strongly} generated by the eight fields $F, L, G, \bar{G}, X, \bar{X}, Y, \bar{Y}$. Moreover, $L$ is  a Virasoro element of central charge $9$, $F$ is primary of weight one, $G, \bar{G}, X, \bar{X}$ are primary of weight 
$\frac{3}{2}$, and $Y, \bar{Y}$ are primary of weight $2$ \cite{OI}. The fields $F,G,\bar G, L$  generate a copy of the $N=2$ superconformal algebra with central charge $c=9$:
\begin{equation}\label{n=2svir}
\begin{split}
F(z) F(w) &\sim 3 (z-w)^{-2},\qquad G(z) G(w)\sim 0,\qquad  \bar{G}(z)  \bar{G}(w) \sim 0,\\ 
 F(z) G(w)&\sim G(w) (z-w)^{-1},\qquad F(z) \bar{G}(w) \sim -\bar{G}(w)(z-w)^{-1},\\ 
 G(z) \bar{G}(w)  &\sim 3(z-w)^{-3} + F(w) (z-w)^{-2} + (L(w) + \frac{1}{2} \partial F(w) )(z-w)^{-1}.
\end{split}
\end{equation}
The fields $F,X,\bar{X}$ satisfy
\begin{equation*}
\begin{split} 
F(z) X(w) &\sim 3 X(w) (z-w)^{-1},\ \ \ \ \ \ F(z) \bar{X}(w) \sim -3 \bar{X}(w) (z-w)^{-1},\\ 
X(z) \bar{X}(w) &\sim -(z-w)^{-3} -F(w)(z-w)^{-2} -\frac{1}{2} \bigg(:F(w)F(w): + \partial F(w) \bigg)(z-w)^{-1}.
\end{split}
\end{equation*}
The additional operator product relations of $F,G,\bar G$ with $X,\bar X, Y, \bar Y$ are
\begin{equation*}
\begin{split}
F(z) Y(w) &\sim 2 Y(w) (z-w)^{-1},\qquad F(z) \bar{Y}(w) \sim -2 \bar{Y}(w) (z-w)^{-1},\\ 
G(z) X(w) &\sim 0,\qquad \bar{G}(z) X(w) \sim 2 Y(w) (z-w)^{-1},\\
\bar{G}(z) \bar{X}(w) &\sim 0, \qquad G(z) \bar{X}(w) \sim 2 \bar{Y}(w)(z-w)^{-1},\\
 G(z) Y(w)&\sim \frac{3}{2} X(w)(z-w)^{-2}+ \frac{1}{2} \partial X(w) (z-w)^{-1},\qquad G(z) \bar{Y}(w)\sim 0,\\
\bar{G}(z) \bar{Y}(w)&\sim \frac{3}{2} \bar{X}(w)(z-w)^{-2}+ \frac{1}{2} \partial \bar{X}(w) (z-w)^{-1}, \qquad \bar G(z) Y(w)\sim 0.
\end{split}
\end{equation*}
The remaining operator product relations of $X,\bar X,Y,\bar Y$ are
\begin{equation*}
\begin{split}
Y(z) &\bar{Y}(w) \sim -\frac{3}{4} (z-w)^{-4} -\frac{1}{2} F(w) (z-w)^{-3} -\\
&\frac{1}{4} \bigg( L(w) + \partial F(w) + \frac{1}{2} :F(w)F(w):\bigg) (z-w)^{-2}+\\ 
&\frac{1}{4} \bigg(:G(w) \bar{G}(w): -:L(w) F(w): - \partial L(w) -\frac{1}{4} \partial \big(:F(w) F(w):\big)\bigg) (z-w)^{-1},\\
X(z) &\bar{Y}(w) \sim -\frac{1}{2} G(w)(z-w)^{-2} - \frac{1}{2} \bigg(:G(w)F(w): + \partial G(w) \bigg) (z-w)^{-1},\\
\bar{X}(z) &Y(w) \sim -\frac{1}{2} \bar{G}(w)(z-w)^{-2} - \frac{1}{2} \bigg(-:\bar{G}(w)F(w): + \partial \bar{G}(w) \bigg) (z-w)^{-1},\\
X(z) &Y(w) \sim 0,\ \ \ \ \ \ \ \bar{X}(z) \bar{Y}(w)\sim 0.\\
\end{split}
\end{equation*} Finally, the following normally ordered polynomial relations hold in $\cO$:
\begin{equation} \label{odakerelations} \partial X =\ :FX:,\ \ \ \ \ \ \partial \bar{X} = - :F\bar{X}:,\ \ \ \ \ \ :YY:\ = 0,\ \ \ \ \ \ :\bar{Y} \bar{Y}:\ = 0.\end{equation}

Odake uses the physics notation for the operators of a field, that is $A(n)=A_{n-h_A+1}$ where $h_A$ is the conformal dimension of the field $A(z)$. 
Clearly $\cO$ is a highest-weight module over the Lie algebra generated by the modes $\{A_{n}|\ n\in \mathbb{Z},\ A\in\{F,L,G,\bar{G},X,\bar{X},Y,\bar{Y}\}$, with highest-weight vector $|0\rangle$ satisfying
\begin{equation*}
 A_n|0\rangle=0,\qquad \text{if}\ n>0\ \text{for all}\ A\in\{F,L,G,\bar{G},X,\bar{X},Y,\bar{Y}\}.
\end{equation*}
Similar, we call $H_{h,m}$ a highest-weight module of the $N=2$ superconformal
algebra with $c=9$ if it contains a highest-weight vector $|h,m\rangle$ satisfying
\begin{equation*}
\begin{split}
L_0|h,m\rangle&=h|h,m\rangle,\qquad F_0|h,m\rangle=m|h,m\rangle,\\
 A_n|h,m\rangle&=0,\qquad \text{if}\ n>0\ \text{for all}\ A\in\{F,L,G,\bar{G}\}
\end{split}
\end{equation*}
Odake shows the following (Theorem 3 in \cite{OII}):
\begin{thm}
The algebra $\cO$ decomposes into highest-weight modules of the $N=2$ superconformal algebra with central charge $c=9$ as
\begin{equation}\label{eq:odakedecomp}
\cO=\bigoplus_{m\in\mathbb Z_{>0}} H_{m^2+m-1/2,2m+1} \oplus H_{0,0} \oplus  \bigoplus_{m\in\mathbb Z_{<0}} H_{m^2-m-1/2,2m-1}
\end{equation}
and the highest-weight vector of $H_{m^2+m-1/2,2m+1}$ is 
\begin{equation*}
 |m^2+m-1/2,2m+1\rangle=Y_{-2m}\cdots Y_{-6}Y_{-4}X_{-3/2}|0\rangle 
\end{equation*}
for $m>0$ and 
\begin{equation*}
 |m^2-m-1/2,2m-1\rangle=\bar Y_{2m}\cdots\bar Y_{-6}\bar Y_{-4}\bar X_{-3/2}|0\rangle 
\end{equation*}
for $m<0$.
\end{thm}
It is known that the vacuum Verma module of the $N=2$ superconformal
algebra with $c=9$ is simple \cite{HR}. There exists a family of automorphisms of this algebra, called spectral flow in physics. 
These are induced automorphisms $\sigma_\alpha, \alpha \in \mathbb Z$ of affine Weyl translations, and they act as
\begin{equation*}
\begin{split}
 \sigma_\alpha(L_n)&=L_n+\alpha F_n+\frac{3}{2}\alpha^2\delta_{n,0},\qquad \sigma_\alpha(F_n)=F_n+3\alpha\delta_{n,0}\\
\sigma_{\alpha}(G_n)&=G_{n+\alpha},\qquad \sigma_{\alpha}(\bar G_n)=\bar G_{n-\alpha}.
\end{split}
\end{equation*}
These can be used to define twisted modules $\sigma^*_\alpha(M)$ for a given module $M$. Each vector $v\in M$ is mapped to $\sigma^*_\alpha(v)$ and the twisted
algebra action is
\begin{equation*}
A\sigma^*_\alpha(v)=\sigma^*_\alpha(\sigma_{-\alpha}(A)v) 
\end{equation*}
for all $A$ in the $N=2$ superconformal algebra of central charge $c=9$. 
The twisted module of an irreducible module is itself irreducible. In particular, a twisted highest-weight vector
satisfying 
\begin{equation*}
A\sigma^*_\alpha(|0\rangle)=\sigma^*_\alpha(\sigma_{-\alpha}(A)|0\rangle) 
\end{equation*}
for all $A$ in the $N=2$ superconformal algebra of central charge $c=9$ must be an element of an irreducible module.
Exactly this equation is satisfied by
\begin{equation}\label{eq:latticestate}
X_{3/2+3\alpha}\cdots X_{-15/2}X_{-9/2}X_{-3/2}|0\rangle 
\end{equation}
for $\alpha<0$. For $\alpha>0$, replace $X$ by $\bar X$ and replace the subscript $3/2+3\alpha$ by $3/2-3\alpha$.
For $\alpha = -m$, the twisted highest-weight vector is related to the highest-weight vector $ |m^2+m-1/2,2m-1\rangle$ by
\begin{equation*}
\begin{split}
\bigl(\bar G_{1/2}\bar G_{3/2}\cdots\bar G_{3/2-\alpha}\bigr)X_{3/2+3\alpha}\cdots X_{-15/2}X_{-9/2}X_{-3/2}|0\rangle 
&=2^{-\alpha-1} Y_{2\alpha}\cdots Y_{-6} Y_{-4} X_{-3/2}|0\rangle\\
 &= 2^{-\alpha-1}|\alpha^2-\alpha-1/2,-2\alpha-1\rangle .
\end{split}
\end{equation*}
Hence we can conclude that $H_{m^2+m-1/2,2m+1}\cong \sigma^*_{-m}(H_{0,0})$ for $m>0$. 
Interchanging barred and unbarred quantities one similarly arrives at $H_{m^2-m-1/2,2m-1}\cong \sigma^*_{-m}(H_{0,0})$ for $m<0$. This means that \eqref{eq:odakedecomp} is a decomposition of Odake's algebra $\cO$ into irreducible modules of 
the $N=2$ superconformal algebra with central charge $c=9$. 
Furthermore, the Heisenberg modules (for the field $F(z)$) with highest-weight state \eqref{eq:latticestate}
and their barred counterparts generate a simple lattice vertex algebra, namely the simple $N=2$ superconformal algebra with central charge $c=1$. Combining these observations, we obtain
\begin{thm}
$\cO$ is a simple vertex algebra. 
\end{thm}

Finally, we will need the character formula for $\cO$, which was computed by Odake in \cite{OII}.
\begin{equation*}\label{char}
\begin{split}
\text{ch}[\cO](z;q)&=\text{tr}_{V_0}(q^{L(0)}z^{F(0)})\\
&=\prod_{n=1}^\infty \frac{(1+zq^{n-\frac{1}{2}})(1+z^{-1}q^{n-\frac{1}{2}})}{(1-q^n)^2}\sum_{m\in\mathbb Z}q^{m^2}z^{2m}-q^{m^2+m+\frac{1}{2}}z^{2m+1}.
\end{split}
\end{equation*} 
\end{example}

\subsection{The commutant construction}

Let $\cV$ be a vertex algebra and let $\cA$ be a subalgebra of $\cV$. The commutant of $\cA$ in $\cV$, denoted by $\text{Com}(\cA,\cV)$, is the subalgebra of vertex operators $v\in\cV$ such that $[a(z),v(w)] = 0$ for all $a\in\cA$. Equivalently, $a(z)\circ_n v(z) = 0$ for all $a\in\cA$ and $n\geq 0$. We regard $\text{Com}(\cA,\cV)$ as the algebra of  invariants in $\cV$ under the action of $\cA$. If $\cA$ is a homomorphic image of an affine vertex algebra $V_k(\gg,B)$, $\text{Com}(\cA,\cV)$ is just the invariant space $\cV^{\gg[t]}$. It is also natural to study the double commutant $\text{Com}(\text{Com}(\cA,\cV),\cV)$, which always contains $\cA$. If $\cA = \text{Com}(\text{Com}(\cA,\cV),\cV)$, we say that $\cA$ and $\text{Com}(\cA,\cV)$ form a {\it Howe pair} inside $\cV$. Since $$\text{Com}(\text{Com}(\text{Com}(\cA,\cV),\cV),\cV) = \text{Com}(\cA,\cV),$$ a subalgebra $\cB$ is a member of a Howe pair if and only if $\cB = \text{Com}(\cA,\cV)$ for some $\cA$.

Let $\gg$ be a simple, finite-dimensional complex Lie algebra and $V$ be an $n$-dimensional linear representation of $\gg$ via $\rho:\gg\rightarrow \text{End}(V)$. Associated to $\rho$ is the bilinear form $B(\xi,\eta) = \text{tr}_V(\rho(\xi) \rho(\eta))$. There is a vertex algebra homomorphism \begin{equation*} \label{defthetae}\tau_{\cE}: V_1(\gg,B)\rightarrow \cE,\ \ \ \ \ \ \ \tau_{\cE}(X^{\xi}) =  \Theta^{\xi}_{\cE} = \sum_{i=1}^n :b^{\rho(\xi)(x_i)} c^{x'_i}:.\end{equation*} There is a similar homomorphim \begin{equation*} \label{defthetas} \tau_{\cS}:V_{-1}(\gg,B)\ra \cS,\ \ \ \ \ \ \ \tau_{\cS}(X^{\xi}) = \Theta^{\xi}_{\cS} = -\sum_{i=1}^n :\beta^{\rho(\xi)(x_i)} \gamma^{x'_i}:.\end{equation*}
There is the combined action \begin{equation*} \label{defthetaw} \tau_{\cW}: V_0(\gg)\ra \cW= \cE\otimes \cS,\ \ \ \  \ \ \ \tau_{\cW}(X^{\xi}) = \Theta^{\xi}_{\cW} = \Theta^{\xi}_{\cE} +\Theta^{\xi}_{\cS}.\end{equation*}

Let $\Theta_{\cE}$, $\Theta_{\cS}$, and $\Theta_{\cW}$ denote the subalgebras of $\cE$, $\cS$, and $\cW$ generated by $\{\Theta^{\xi}_{\cE}\}$, $\{\Theta^{\xi}_{\cS}\}$, and $\{\Theta^{\xi}_{\cW}\}$, respectively, for $\xi\in\gg$. We are interested in the commutants $\text{Com}(\Theta_{\cE},\cE)$, $\text{Com}(\Theta_{\cS},\cS)$, and $\text{Com}(\Theta_{\cW},\cW)$, which coincide with $\cE^{\gg[t]}$, $\cS^{\gg[t]}$, and $\cW^{\gg[t]}$, respectively. Since $\cS^{\gg[t]}\subset \cW^{\gg[t]}$, we may define \begin{equation*}\cC(\gg,V) = \text{Com}(\cS^{\gg[t]}, \cW^{\gg[t]}).\end{equation*} We will study $\cC(\gg,V)$ in the special cases where $\gg = \gs\gl_2$ and $V$ is the standard representation $\mathbb{C}^2$ and adjoint representation $\mathbb{C}^3$.

\section{Some general features of $\cE^{\gg[t]}$, $\cS^{\gg[t]}$, $\cW^{\gg[t]}$, and $\cC(\gg,V)$}
First, $\cE^{\gg[t]}$ possesses a Virasoro element $$L_{\cE^{\gg[t]}} = L_{\cE} - \tau_{\cE}( L_{\text{Sug}}),$$ where $L_{\text{Sug}}$ is the Sugawara conformal vector in $V_1(\gg,B)$. There is also a $U(1)$-current  $$F = -\sum_{i=1}^n :b^{x_i} c^{x'_i}:.$$
Similarly, $\cS^{\gg[t]}$ possesses a Virasoro element $$L_{\cS^{\gg[t]}} = L_{\cS} - \tau_{\cS}(L_{\text{Sug}})$$ where $L_{\text{Sug}}$ is the Sugawara vector in $V_{-1}(\gg,B)$, which exists unless $-B$ is the critical level bilinear form. There is also a $U(1)$ current $$H = \sum_{i=1}^n :\beta^{x_i} \gamma^{x'_i}:.$$ If $V$ possesses a $\gg$-invariant, symmetric bilinear form, we may choose a corresponding orthonormal basis $x_1,\dots,x_n$ for $V$ and dual basis $x'_1,\dots,x'_n$ for $V^*$. Then $H$ is part of an action of $V_{-n/2}(\gs\gl_2)$ on $\cS^{\gg[t]}$, given by 
\begin{equation*}\label{slii} 
\begin{split}
X^h\mapsto v^h &= \sum_{i=1}^n :\beta^{x_i} \gamma^{x'_i}:,  \qquad \ X^x\mapsto  v^x =    \frac{1}{2} \sum_{i=1}^n:\gamma^{x'_i} \gamma^{x'_i}:,  \\ 
X^y\mapsto   v^y &=  -\frac{1}{2} \sum_{i=1}^n:\beta^{x_i} \beta^{x_i}:.
\end{split}
\end{equation*} 
The generators $X^x, X^y, X^h$ for $V_{-n/2}(\gs\gl_2)$ are in the usual root basis $x,y,h$ with commutation relations $$[x,y]=h,\ \ \ \ \ \ \ \ \ \ \  [h,x]=2x,\ \ \ \ \ \ \ \ \ \ \  [h,y]=-2y.$$ 
Finally, $\cW^{\gg[t]}$ has a Virasoro element $$L_{\cW^{\gg[t]}}= L_{\cW} - \tau_{\cW}( L_{\text{Sug}}).$$ If $V$ possesses a symmetric invariant form as above, the $V_{-n/2}(\gs\gl_2)$-structure on $\cS^{\gg[t]}$ extends to an action of the affine vertex superalgebra $V_n(\go\gs\gp(2|2))$ on $\cW^{\gg[t]}$. The Lie superalgebra $\go\gs\gp(2|2)$ is generated by four even elements $X,Y,H,E$ and four odd elements $F^{\epsilon\epsilon'}$, where $\epsilon,\epsilon'\in\{\pm\}$, and nonzero relations
\begin{equation*}
\begin{split}
&[X,Y]\, = \,  H \ \,\ \ [H,X]\,=\, 2X\ \ ,\ \ [H,Y]\,=\,-2Y \,,\ \
[E,F^{\epsilon\epsilon'}]\, =\, \epsilon' F^{\epsilon\epsilon'}\ \ ,\\
&[H,F^{\epsilon\epsilon'}]\, =\, \epsilon F^{\epsilon\epsilon'}\ \ ,\ \
[Y,F^{+\epsilon}]\, =\, -F^{-\epsilon}\ \ ,\ \
[X,F^{-\epsilon}]\, =\, -F^{+\epsilon}\ \ , \ \ 
[F^{+-},F^{-+}]\, =\, H+E\,, \\ 
&[F^{+-},F^{++}]\, =\, 2X \ \ , \ \
[F^{--},F^{++}]\, =\, H-E \ \ , \ \  [F^{--},F^{-+}]\, =\, -2Y\, .
\end{split}
\end{equation*}
A nondegenerate consistent graded-symmetric invariant bilinear form is given by
\begin{equation*}
\begin{split}
&B(H,H)\,=\, -1 \ \ , \ \ B(X,Y)\,=\, -\frac{1}{2} \ \ , \ \ B(E,E) \, = \, 1\,,\\
&B(F^{+-},F^{-+}) \,= \, -1 \ \ , \ \ B(F^{--},F^{++}) \,= \, 1\, .
\end{split}
\end{equation*}  

The Lie superalgebra $\go\gs\gp(2|2)$ is simple and the bilinear form $B$ is unique up to a scalar multiple. Define vertex operators $$Q^{\beta b} = \sum_{i=1}^n :\beta^{x_i} b^{x_i},\ \ \ \ Q^{\beta c} = \sum_{i=1}^n :\beta^{x_i} c^{x'_i}:,\ \ \ \  Q^{\gamma b} = \sum_{i=1}^n :\gamma^{x'_i} b^{x_i}:, \ \ \ \ Q^{\gamma c} = \sum_{i=1}^n :\gamma^{x'_i} c^{x'_i}:,$$ which are easily seen to lie in $\cW^{\gg[t]}$. Then the $V_{-n/2} (\gs\gl_2)$ structure on $\cS^{\gg[t]}$ extends to an action of the affine vertex superalgebra $V_n(\go\gs\gp(2|2))$ on $\cW^{\gg[t]}$, where the generators are the quadratics $v^x, v^y, v^h, F, Q^{\beta b}, Q^{\beta c}, Q^{\gamma b}, Q^{\gamma c}$.

There are also a few properties of $\cC(\gg,V)$ which are worth pointing out. First, $\cE^{\gg[t]}$ is always a subalgebra of $\cC(\gg,V)$. In general, $\cC(\gg,V)$ is bigger than $\cE^{\gg[t]}$ and contains fields that depend on $\beta^x, \gamma^{x'}$. There is a Virasoro element in  $\cC(\gg,V)$ given by $$L = L_{\cW^{\gg[t]}} - L_{\cS^{\gg[t]}}.$$

\begin{thm} \label{howepair} $Com(\cC(\gg,V), \cW^{\gg[t]}) = \cS^{\gg[t]}$, so $\cS^{\gg[t]}$ and $\cC(\gg,V)$ form a Howe pair inside $\cW^{\gg[t]}$.
\end{thm}

\begin{proof} Clearly $\cS^{\gg[t]} \subset \text{Com}(\cC(\gg,V),\cW^{\gg[t]})$. Since $L_{\cW^{\gg[t]}}$ and $L$ are conformal structures on $\cW^{\gg[t]}$ and $\cC(\gg,V)$, respectively, $$L_{\cW^{\gg[t]}}- L = L_{\cS^{\gg[t]}}$$ is a conformal structure on $\text{Com}(\cC(\gg,V),\cW^{\gg[t]})$. Any field in $\text{Com}(\cC(\gg,V),\cW^{\gg[t]})$ therefore cannot depend on the generators $b^x, c^{x'}$ of $\cE$ and their derivatives. We conclude that $\text{Com}(\cC(\gg,V),\cW^{\gg[t]})\subset \cS^{\gg[t]}$. \end{proof}

\section{The case where $\gg = \gs\gl_2$ and $V$ is the adjoint module}
In this section, we consider $\cE^{\gg[t]}$, $\cS^{\gg[t]}$, $\cW^{\gg[t]}$, and $\cC(\gg,V)$ in the case where $\gg = \gs\gl_2$ and $V$ is the adjoint representation $\mathbb{C}^3$. We work in the standard root basis $x,y,h$ for $\gs\gl_2$. The generators of $\Theta_{\cW}$ in this basis are
\begin{equation*}
 \begin{split}
\Theta^x_{\cW} &= \Theta^x_{\cE}+\Theta^x_{\cS},\ \ \ \ \Theta^x_{\cS} = 2 :\beta^x \gamma^{h'}: - :\beta^h \gamma^{y'}: ,\ \ \ \ \Theta^x_{\cE} =   - 2:b^x c^{h'}: + :b^h, c^{y'}:,\\
\Theta^y_{\cW} &= \Theta^y_{\cE}+\Theta^y_{\cS},\ \ \ \ \ \Theta^y_{\cS} = -2 :\beta^y \gamma^{h'}: + :\beta^h \gamma^{x'}:  ,\ \ \ \ \Theta^y_{\cE} =   2 :b^y c^{h'}: - :b^h c^{x'}:,\\
\Theta^h_{\cW} &= \Theta^h_{\cE}+\Theta^h_{\cS},\ \ \ \ \Theta^h_{\cS} = -2 :\beta^x \gamma^{x'}:  + 2 : \beta^y \gamma^{y'}: ,\ \ \ \ \Theta^h_{\cE} =   2 :b^x c^{x'}: - 2 :b^y c^{y'}:.
\end{split}
\end{equation*}

\begin{thm} \label{spiece}$\cS^{\gs\gl_2[t]}$ is isomorphic to $V_{-3/2} (\gs\gl_2)$ and is generated by 
\begin{equation*}
\begin{split}
 v^h &=\  :\beta^h \gamma^{h'}: + :\beta^x \gamma^{x'}: +:\beta^y\gamma^{y'}: ,\\
 v^x &= \frac{1}{2}\big(:\gamma^{h'} \gamma^{h'}: + :\gamma^{x'} \gamma^{y'}:\big),\\ 
 v^y &= -\frac{1}{2}\big(:\beta^h \beta^h:+ 4:\beta^x \beta^{y}:\big).
\end{split}
\end{equation*} \end{thm}

This was proven in \cite{L}. In fact, $\cS^{\gs\gl_2[t]}$ and $\Theta_{\cS}$ form a Howe pair inside $\cS$ (see \cite{LL}).

Next, we recall the Friedan-Martinec-Shenker bosonization of fermions \cite{FMS}. 
Let $\mathfrak{h}$ be the Heisenberg Lie algebra with generators $j(n)$, $n\in \mathbb{Z}$, and $\kappa$, satisfying $[j(n),j(m)] = n \delta_{n+m,0}\kappa$.
We set the central element to one, $\kappa=1$. 
The field $j(z)= \sum_{n\in \mathbb{Z}} j(n) z^{-n-1}$ satisfies the operator product $$j(z)j(w)\sim (z-w)^{-2},$$ and generates a level one Heisenberg vertex algebra $\cH$. Define the {\it free bosonic scalar field}
$$\phi(z) = q+ j(0) \ln z - \sum_{n\neq 0} \frac{j(n)}{n} x^{-n},$$ where $q$ satisfies $[j(n),q] = \delta_{n,0}$. Clearly $\partial \phi(z) = j(z)$, and we have the operator product $$\phi(z)\phi(w)\sim \ln(z-w).$$
Given $\alpha\in \mathbb{C}$, let $\cH_{\alpha}$ denote the irreducible representation of $\mathfrak{h}$ generated by the vacuum vector $v_{\alpha}$ satisfying \begin{equation*} j(n) v_{\alpha}= \alpha \delta_{n,0} v_{\alpha},~~~ n\geq 0.\end{equation*} Given $\eta\in \mathbb{C}$, the operator $e^{\eta q}(v_{\alpha}) = v_{\alpha+\eta}$, so $e^{\eta q}$ maps $\cH_{\alpha}\ra \cH_{\alpha+\eta}$. Define the vertex operator
$$\Gamma_{\eta}(z) = e^{\eta \phi(z)} = e^{\eta q} z^{\eta \alpha} \text{exp} \bigg(\eta \sum_{n>0} j(-n) \frac{z^n}{n}\bigg) \text{exp}\bigg(\eta\sum_{n<0} j(-n) \frac{z^n}{n}\bigg).$$ The $\Gamma_{\eta}$ satisfy the operator products
$$j(z) \Gamma_{\eta}(w) \sim \eta \Gamma_{\eta}(w)(z-w)^{-1} + \frac{1}{\eta} \partial \Gamma_{\eta}(w),$$
$$ \Gamma_{\eta}(z) \Gamma_{\nu}(w) \sim (z-w)^{\eta\nu} : \Gamma_{\eta}(z) \Gamma_{\nu}(w):.$$
If we take $L$ to be the one-dimensional lattice $\mathbb{Z}$ with generator $\eta$, the pair of (fermionic) fields $\Gamma_{\eta}, \Gamma_{-\eta}$ generate the lattice vertex algebra $V_L$. The state space of $V_L$ is just $\sum_{n\in\mathbb{Z}} \cH_n = \cH\otimes_{\mathbb{C}} L$. It follows that $$\Gamma_{\eta}(z) \Gamma_{-\eta}(w) \sim (z-w)^{-1}, \qquad \Gamma_{-\eta}(z) \Gamma_{\eta}(w) \sim (z-w)^{-1},$$ $$\Gamma_{\eta}(z) \Gamma_{\eta}(w) \sim 0,\qquad \Gamma_{-\eta}(z) \Gamma_{-\eta}(w)\sim 0,$$ so $V_L$ is isomorphic to the $bc$-system of rank $1$. 

\begin{thm} \label{commbc} $\cE^{\gs\gl_2[t]}$ is strongly generated by
\begin{equation*} 
\begin{split}
F &= -: b^h c^{h'}: - : b^x c^{x'}: - : b^y c^{y'}:,\\
C^{bbb} &=\ :b^x b^y b^h:,\\
C^{ccc} &=\ :c^{x'} c^{y'} c^{h'}:.
\end{split}
 \end{equation*} \end{thm}

\begin{proof} 
We use the above bosonization to express $\cE$ as a lattice vertex algebra $V_L$ where $L \cong \mathbb{Z}^3$, with generators $X,Y,H$. Let $\phi_X, \phi_Y, \phi_H$ be free bosons satisfying
$$\phi_X(z)\phi_X(w) \sim \ln(z-w), \qquad \phi_Y(z)\phi_Y(w) \sim \ln(z-w),\qquad \phi_H(z)\phi_H(w) \sim \ln(z-w).$$ We have an isomorphism $\cE \cong V_L$ given by
$$b^x \mapsto \Gamma_X ,\qquad  c^{x'} \mapsto  \Gamma_{-X},\qquad b^y \mapsto \Gamma_Y ,\qquad  c^{y'} \mapsto \Gamma_{-Y}, \qquad b^h \mapsto \Gamma_H,\qquad  c^{h'} \mapsto \Gamma_{-H}.$$
Under this isomorphism, 
\begin{equation*}
\begin{split}
\Theta^x_{\cE}\, &\mapsto \, -2 \Gamma_{X-H} +\Gamma_{H-Y},\\
\Theta^y_{\cE}\, &\mapsto \, 2 \Gamma_{Y-H}- \Gamma_{H-X}\,,\\
\Theta^h_{\cE}\, & \mapsto\, 2 \partial \phi_X-2\partial \phi_Y\,.
\end{split}
\end{equation*}
We change orthogonal bases as follows
\begin{equation*}
A\, =\, X+Y+H ,\quad \quad B\, =\, X+Y-2H ,\quad \quad C\,=\,X-Y\, .
\end{equation*}
Now, we see that the rank one lattice vertex algebra generated by $\Gamma_A$ and $\Gamma_{-A}$ lies in $\cE^{\gs\gl_2[t]}$, since
\begin{equation*}
\begin{split}
\Theta^x_{\cE}\, &=\, :\Gamma_{C/2}(-2 \Gamma_{B/2}+ \Gamma_{-B/2}):\,,\\
\Theta^y_{\cE}\, &=\, :\Gamma_{-C/2}(2 \Gamma_{B/2}- \Gamma_{-B/2}):\,,\\
\Theta^h_{\cE}\, &=\, 2 \partial \phi_C\,.
\end{split}
\end{equation*}
Moreover, any field depending on $C$ does not commute with $\Theta^h_{\cE}$. It remains to show that any field that depends on $B$ does not lie in $\cE^{\gs\gl_2[t]}$. We compute the normally ordered product
\begin{equation*}
\frac{1}{6}\bigl(:\Theta^x_{\cE}\Theta^y_{\cE}:-\frac{1}{2}\partial \Theta^h_{\cE}-\frac{1}{8}:\Theta^h_{\cE}\Theta^h_{\cE}:\bigr) \,=\, \frac{1}{12}:\partial \phi_B\partial \phi_B:\, = \, T_B\, .
\end{equation*}
Here $T_B$ is the Virasoro field of central charge one in the rank one lattice vertex algebra generated by $\Gamma_B$ and $\Gamma_{-B}$. Any field in $\cE^{\gs\gl_2[t]}$ must commute with $T_B$, and this is only true of fields not depending on $B$.
\end{proof}

Unfortunately, we cannot describe $\cW^{\gs\gl_2[t]}$ using any of our methods. In addition to the elements $v^x,v^y,v^h,F,C^{bbb},C^{ccc}$ given above, the elements
\begin{equation}\label{geniv} 
\begin{split}
 Q^{\gamma b}  &=\ :\gamma^{h'} b^h: + :\gamma^{x'} b^x: + :\gamma^{y'} b^y:, \ \ \ \ \ \ \ \ \ \ Q^{\beta b} =\ :\beta^h b^h:+ 2 :\beta^x b^y:+2:\beta^y b^x:,\\
 Q^{\gamma c}  &=\ :\gamma^{h'} c^{h'}: + \frac{1}{2}:\gamma^{x'} c^{y'}: +\frac{1}{2} :\gamma^{y'} c^{x'}:, \ \ \ \ \ \ \ \ \ \ Q^{\beta c} =\ :\beta^h c^{h'}: +  : \beta^x c^{x'}:+ : \beta^y c^{y'}:,
\end{split}
\end{equation}
are part of the $V_3(\go\gs\gp(2|2))$-structure. There are six additional elements of $\cW^{\gs\gl_2[t]}$ that we can write down using using Weyl's first fundamental theorem of invariant theory for $\gs\go_3$.
\begin{equation}\label{genvi}
\begin{split}
 G&=\  :\beta^h \gamma^{x'} c^{y'}: - :\beta^h \gamma^{y'} c^{x'}:  + 2 :\beta^x \gamma^{h'} c^{x'}: - 2 :\beta^x \gamma^{x'} c^{h'}:  - 2 :\beta^y \gamma^{h'} c^{y'}:  \\
 &\qquad + 2 :\beta^y \gamma^{y'} c^{h'}: -:b^h c^{x'} c^{y'}: +  2 :b^x c^{x'} c^{h'}: - 2 :b^y c^{y'} c^{h'}:,\\
\bar{G} &= \frac{1}{2} \big( -:\beta^h \gamma^{x'} b^x: + : \beta^h \gamma^{y'} b^y:  - 2 : \beta^x \gamma^{h'} b^y: + :\beta^x \gamma^{x'} b^h:+ 2 :\beta^y \gamma^{h'}  b^x: \\
 &\qquad - :\beta^y \gamma^{y'} b^h: + :b^x b^h c^{x'}: - 2 :b^x b^y c^{h'}:  - :b^y b^h c^{y'}: \big),\\
 C^{\gamma bb} &= - : \gamma^{h'} b^x b^y: + \frac{1}{2} :\gamma^{x'} b^x b^h:  - \frac{1}{2} :\gamma^{y'} b^y b^h:,\\
 C^{\beta bb} &=\ :\beta^h b^x b^y: + :\beta^x b^y b^h: - :\beta^y b^x b^h:,\\
 C^{\gamma cc} &= -:\gamma^{h'} c^{x'} c^{y'} : -:\gamma^{x'} c^{y'} c^{h'}: + :\gamma^{y'} c^{x'} c^{h'}:,\\ 
 C^{\beta cc} &=\ :\beta^h c^{x'} c^{y'}: - 2 :\beta^x c^{x'} c^{h'}: + 2 :\beta^y c^{y'} c^{h'}:.\
\end{split}
 \end{equation}

\begin{remark} $\cW$ coincides with the semi-infinite Weil complex of $\gs\gl_2$ (see \cite{FF}), and the zero mode $G(0)$ is the semi-infinite differential. \end{remark}

Recall that $\cS^{\gs\gl_2[t]}$ has a Virasoro element $L_{\cS^{\gs\gl_2[t]}} = L_{\cS} - \tau_{\cS}(L_{\text{Sug}})$ of central charge zero, which is given by
\begin{equation} \label{virs} L_{\cS^{\gs\gl_2[t]}} =\frac{1}{2}\big(4 :v^x v^y: + :v^h v^h: - \partial v^h\big).\end{equation} Note that $L_{\cS^{\gs\gl_2[t]}}$ is {\it not} the usual Sugawara vector for $V_{-3/2}(\gs\gl_2)$; we have 
\begin{equation*}
\begin{split}
 L_{\cS^{\gs\gl_2[t]}} (z) v^h(w) &\sim 3(z-w)^{-3} + v^h(w) (z-w)^{-2} + \partial v^h(w) (z-w)^{-1},\\
 L_{\cS^{\gs\gl_2[t]}} (z) v^y(w)  &\sim 2v^y(w) (z-w)^{-2} + \partial v^y(w) (z-w)^{-1},\\
 L_{\cS^{\gs\gl_2[t]}} (z) v^x(w) &\sim  \partial v^x(w) (z-w)^{-1},
 \end{split}
 \end{equation*}
which shows that $v^x$ and $v^y$ are primary of weights zero and two, and $v^h$ is quasi-primary of weight one. With this Virasoro element, $\cS^{\gs\gl_2[t]}\hookrightarrow \cW$ is a conformal embedding.

Let $\cA$ denote the subalgebra of $\cW^{\gs\gl_2[t]}$ generated by $v^x,v^y,v^h,F,C^{bbb},C^{ccc}$, together with the vertex operators \eqref{geniv} and \eqref{genvi}. We conjecture that $\cA$ is all of $\cW^{\gs\gl_2[t]}$, but we are unable to prove this at present. We will write down all nontrivial operator products among the generators of $\cA$, and we will see that $\cA$ is {\it strongly} generated by the above collection. We start with the operator products of $v^x, v^y, v^h, F, Q^{\gamma b}, Q^{\beta b}, Q^{\gamma c}, Q^{\beta c}$.
\begin{equation*} 
\begin{split}
 v^x(z) v^y(w) &\sim -\frac{3}{2}(z-w)^{-2} + v^h(w) (z-w)^{-1},\ \ \ \ \ \  v^h(z) v^h(w) \sim -3(z-w)^{-2},\\
 v^h(z) v^x(w) &\sim 2 v^x(w) (z-w)^{-1} ,\ \ \ \ \ v^h(z) v^y(w) \sim -2 v^y(w)(z-w)^{-1},\\
 F(z) F(w) &\sim 3 (z-w)^{-2},\\ 
 F(z) Q^{\gamma c}(w) &\sim Q^{\gamma c}(w)(z-w)^{-1},\ \ \ \ \ F(z) Q^{\gamma b}(w) \sim -Q^{\gamma b}(w)(z-w)^{-1},\\
 F(z) Q^{\beta c}(w) &\sim Q^{\beta c}(w)(z-w)^{-1},\ \ \ \ \ F(z) Q^{\beta b}(w) \sim -Q^{\beta b}(w)(z-w)^{-1},\\
 Q^{\gamma b} (z) Q^{\beta c} (w) &\sim -3 (z-w)^{-2}+ (v^h + F) (z-w)^{-1},\\
 Q^{\beta b} (z)Q^{\gamma c} (w) &\sim 3 (z-w)^{-2} + (v^h -F)(z-w)^{-1},\\
 Q^{\gamma b} (z) Q^{\gamma c} (w) &\sim 2 v^x (z-w)^{-1},\ \ \ \ \ Q^{\beta b} (z) Q^{\beta c} (w) \sim -2 v^y(z-w)^{-1},\\
v^h(z) Q^{\gamma b}(w) &\sim Q^{\gamma b}(w)(z-w)^{-1},\ \ \ \ \ v^y(z) Q^{\gamma b}(w) \sim -Q^{\beta b} (w)(z-w)^{-1},\\
v^h(z) Q^{\beta b} (w) &\sim -Q^{\beta b} (w)(z-w)^{-1},\ \ \ \ \ v^x(z) Q^{\beta b} (w) \sim -Q^{\gamma b}(w)(z-w)^{-1},\\
v^h(z) Q^{\gamma c}(w) &\sim Q^{\gamma c}(w)(z-w)^{-1},\ \ \ \ \  v^y(z) Q^{\gamma c}(w) \sim -Q^{\beta c} (w)(z-w)^{-1},\\
v^h(z) Q^{\beta c} (w) &\sim -Q^{\beta c} (w)(z-w)^{-1},\ \ \ \ \ v^x(z) Q^{\beta c} (w) \sim -Q^{\gamma c}(w)(z-w)^{-1}.
\end{split}
\end{equation*}
These fields obey the operator product relations of the affine vertex superalgebra $V_3(\go\gs\gp(2|2))$.
Next, $C^{\gamma bb}, C^{\beta bb}, C^{bbb}, \bar G$ transform under $V_3(\go\gs\gp(2|2))$ as follows:

\begin{equation*}\label{ospprimaryI}
\begin{split}
 v^h(z) C^{\gamma bb}(w) &\sim C^{\gamma bb}(w)(z-w)^{-1},\ \ \ \ \ v^y(z) C^{\gamma bb}(w) \sim C^{\beta bb}(w)(z-w)^{-1},\\
v^h(z) C^{\beta bb}(w) &\sim -C^{\beta bb}(w)(z-w)^{-1},\ \ \ \ \ v^x(z) C^{\beta bb}(w) \sim  C^{\gamma bb}(w)(z-w)^{-1},\\
F(z) C^{\gamma bb}(w) &\sim - 2C^{\gamma bb}(w)(z-w)^{-1},\ \ \ \ \ F(z) C^{\beta bb}(w) \sim - 2C^{\beta bb}(w)(z-w)^{-1},\\
F(z) C^{bbb}(w) &\sim - 3C^{bbb}(w)(z-w)^{-1},\ \ \ \ \ F(z) \bar{G}(w) \sim - \bar{G}(w)(z-w)^{-1},\\
Q^{\gamma b}(z)C^{\beta bb}(w) &\sim -3 C^{bbb}(w)(z-w)^{-1},\ \ \ \ \ Q^{\beta b}(z) C^{\gamma bb}(w) \sim -3 C^{bbb}(w)(z-w)^{-1},\\
Q^{\beta c}(z)C^{\gamma bb}(w) &\sim \bar{G}(w)(z-w)^{-1} ,\ \ \ \ \ Q^{\gamma c}(z)C^{\beta bb}(w) \sim \bar{G}(w)(z-w)^{-1},\\
\bar{G}(z) Q^{\beta b}(w) &\sim C^{\beta bb}(w) (z-w)^{-1},\ \ \ \ \ \bar{G}(z) Q^{\gamma b}(w) \sim -C^{\gamma bb}(w) (z-w)^{-1}.
\end{split}
\end{equation*}
The transformations of $C^{\gamma cc}, C^{\beta cc}, C^{ccc}, G$ under $V_3(\go\gs\gp(2|2))$ are given by
\begin{equation*}\label{ospprimaryII}
\begin{split}
 v^h(z) C^{\gamma cc}(w) &\sim C^{\gamma cc}(w)(z-w)^{-1},\ \ \ \ \  v^y(z) C^{\gamma cc}(w) \sim C^{\beta cc}(w)(z-w)^{-1},\\
v^h(z) C^{\beta cc}(w) &\sim -C^{\beta cc}(w)(z-w)^{-1},\ \ \ \ \ v^x(z) C^{\beta cc}(w) \sim  C^{\gamma cc}(w)(z-w)^{-1},\\
 F(z) C^{\gamma cc}(w) &\sim 2 C^{\gamma cc}(w)(z-w)^{-1},\ \ \ \ \ F(z) C^{\beta cc}(w) \sim 2 C^{\beta cc}(w)(z-w)^{-1},\\
 F(z) C^{ccc}(w) &\sim 3 C^{ccc}(w)(z-w)^{-1},\ \ \ \ \  F(z) G(w) \sim G(w)(z-w)^{-1},\\
Q^{\gamma c}(z)C^{\beta cc}(w) &\sim -3 C^{ccc}(w)(z-w)^{-1},\ \ \ \ \ Q^{\beta c}(z)C^{\gamma cc}(w) \sim -3 C^{ccc}(w)(z-w)^{-1},\\
Q^{\gamma b}(z)C^{\beta cc}(w) &\sim G(w)(z-w)^{-1},\ \ \ \ \ Q^{\beta b}(z)C^{\gamma cc}(w) \sim G(w)(z-w)^{-1},\\
G(z) Q^{\beta c}(w) &\sim C^{\beta cc}(w) (z-w)^{-1},\ \ \ \ \ G(z) Q^{\gamma c}(w) \sim -C^{\gamma cc}(w) (z-w)^{-1}.
\end{split}
\end{equation*}
Finally, the operator products of $C^{\gamma cc}, C^{\beta cc}, C^{ccc}, G$ with $C^{\gamma bb}, C^{\beta bb}, C^{bbb}, \bar G$ can be expressed in terms of fields in the image of $V_3(\go\gs\gp(2|2))$ as follows:
\begin{equation*}
 \begin{split}
C^{\gamma bb}(z)& C^{\gamma cc}(w) \sim -2 v^x(w)(z-w)^{-2} +\\
&\qquad\qquad\qquad \bigg(:Q^{\gamma b}(w) Q^{\gamma c}(w): + 2 :v^x(w) F(w): -2 \partial v^x(w) \bigg) (z-w)^{-1},\\
C^{\beta bb}(z)& C^{\beta cc}(w) \sim 2 v^y(w)(z-w)^{-2} +\\\
&\qquad\qquad\qquad \bigg(:Q^{\beta b}(w) Q^{\beta c}(w): -2 :v^y(w) F(w): +2 \partial v^y(w) \bigg) (z-w)^{-1},\\
C^{\gamma bb}(z)& C^{\beta cc}(w) \sim -3 (z-w)^{-3} + \big( v^h(w) + 2F(w)\big) (z-w)^{-2}-\\
&\bigg(:Q^{\beta b}(w) Q^{\gamma c}(w): +  :v^h(w) F(w): +\frac{1}{2}:F(w)F(w): -\frac{1}{2}\partial F(w)- \partial v^h(w) \bigg) (z-w)^{-1},\\
C^{\gamma cc}(z)& C^{\beta bb}(w) \sim -3 (z-w)^{-3} + \big( v^h(w) - 2F(w)\big) (z-w)^{-2}-\\
&\bigg(:Q^{\beta c}(w) Q^{\gamma b}(w): -  :v^h(w) F(w): +\frac{1}{2}:F(w)F(w): +\frac{1}{2}\partial F(w)- \partial v^h(w) \bigg) (z-w)^{-1},\\
\end{split}
\end{equation*}
\begin{equation*}
 \begin{split} 
&C^{\gamma cc}(z) C^{bbb}(w) \sim Q^{\gamma b}(w)(z-w)^{-2} + \big(:Q^{\gamma b}(w) F(w): \big) (z-w)^{-1},\\
&C^{\beta cc}(z) C^{bbb}(w) \sim -Q^{\beta b}(w)(z-w)^{-2} - \big(:Q^{\beta b}(w) F(w): \big) (z-w)^{-1},\\
&C^{\gamma bb}(z) C^{ccc}(w) \sim Q^{\gamma c}(w)(z-w)^{-2} - \big(:Q^{\gamma c}(w) F(w): \big) (z-w)^{-1},\\
\end{split}
\end{equation*}
\begin{equation*}
 \begin{split}
&C^{\beta bb}(z) C^{ccc}(w) \sim -Q^{\beta c}(w)(z-w)^{-2} + \big(:Q^{\beta c}(w) F(w): \big) (z-w)^{-1},\\
&C^{ccc}(z) C^{bbb}(w)  \sim -(z-w)^{-3} -F(w)(z-w)^{-2} -\frac{1}{2} \bigg(:F(w)F(w): + \partial F(w) \bigg)(z-w)^{-1},\\
&G(z) C^{bbb}(w) \sim - :Q^{\gamma b}(w) Q^{\beta b}(w): (z-w)^{-1},\\
&G(z) C^{\gamma bb}(w) \sim Q^{\gamma b}(w)(z-w)^{-2} + \big(: v^h Q^{\gamma b}: -2 :v^x Q^{\beta b}:~ -\partial Q^{\gamma b}\big)(z-w)^{-1},\\
 &G(z) C^{\beta bb}(w) \sim -Q^{\beta b} (w)(z-w)^{-2} + \big(2: v^y Q^{\gamma b}:~+~ :v^h Q^{\beta b} :~ + \partial Q^{\beta b} \big)(z-w)^{-1},\\
&\bar{G}(z) C^{ccc}(w) \sim :Q^{\beta c}(w) Q^{\gamma c}(w): (z-w)^{-1},\\
 &\bar{G}(z) C^{\gamma cc}(w) \sim  Q^{\gamma c} (z-w)^{-2} - \big(2 :v^x(w) Q^{\beta c}(w): -v^h(w) Q^{\gamma c}(w): + \partial Q^{\gamma c}(w) \big) (z-w)^{-1},\\
&\bar{G}(z) C^{\beta cc}(w) \sim  -Q^{\beta c} (z-w)^{-2} + \big(2 :v^y(w) Q^{\gamma c}(w): +v^h(w) Q^{\beta c}(w): + \partial Q^{\beta c}(w) \big) (z-w)^{-1},\\
\end{split}
\end{equation*}
\begin{equation}\label{gbarg}
 \begin{split}
&G(z) \bar{G}(w) \sim 3 (z-w)^{-3} + F(w)(z-w)^{-2} + \bigg(- 4 :v^x(w) v^y(w): + :v^h(w) v^h(w): \\
&\qquad + :Q^{\gamma b}(w)Q^{\beta c}(w): - :Q^{\beta b} (w)Q^{\gamma c}(w): + \frac{1}{2} :F(w) F(w): +2 \partial v^h(w) -\frac{1}{2} \partial F(w) \bigg) (z-w)^{-1}.
\end{split}
\end{equation}

\section{The Odake algebra structure}

Recall that $\cW^{\gs\gl_2[t]}$ has a conformal vector $L_{\cW^{\gs\gl_2[t]}} = L_{\cW} - \tau_{\cW}(L_{\text{Sug}})$ of central charge zero, which is given by $$L_{\cW^{\gs\gl_2[t]}} =\ :\beta^h \partial \gamma^{h'}: +:\beta^x \partial \gamma^x: + :\beta^y \partial \gamma^{y'}: - :b^h \partial c^{h'}: -:b^x \partial c^x: - :b^y \partial c^{y'}: $$ $$+ \frac{1}{4} \big( :\Theta^x_{\cW} \Theta^y_{\cW}: + :\Theta^y_{\cW} \Theta^x_{\cW}: + \frac{1}{2}  :\Theta^h_{\cW} \Theta^h_{\cW}: \big).$$ In fact, $L_{\cW^{\gs\gl_2[t]}}$ lies in $\cA$ and hence can be realized as a normally ordered polynomial in the generators of $\cA$ as follows: 
$$L_{\cW^{\gs\gl_2[t]}}  = G\circ_0 \bar{G} + L_{\cS^{\gs\gl_2[t]}},$$ where $G\circ_0 \bar{G}$ is given by (\ref{gbarg}) and $L_{\cS^{\gs\gl_2[t]}}$ is given by (\ref{virs}). Define 
\begin{equation*} 
L = G\circ_0 \bar{G} - \frac{1}{2} \partial F = L_{\cW^{\gs\gl_2[t]}}  -  L_{\cS^{\gs\gl_2[t]}}  - \frac{1}{2} \partial F .
\end{equation*} 
It is straightforward to check that $L$ is a Virasoro element in $\cC(\gs\gl_2,\mathbb{C}^3)$ with central charge $9$, $G, \bar{G}$ are both primary of weight $\frac{3}{2}$ with respect to $L$, and $F$ is primary of weight $1$. Moreover, \eqref{n=2svir} is satisfied, so $G,\bar{G}, F, L$ represent a copy of the $N=2$ superconformal vertex algebra inside $\cC(\gs\gl_2,\mathbb{C}^3)$. Finally, we define elements
\begin{equation*} 
\begin{split}
X &= C^{ccc},\qquad\qquad\qquad\qquad\ \ \ \ \ \ \ \ \ \ \ \ \bar{X} = C^{bbb},\\
 Y &= \frac{1}{2} \bar{G}(0)(X) = \frac{1}{2} :Q^{\beta c} Q^{\gamma c}:,\ \ \ \ \ \ \ \  \bar{Y} = \frac{1}{2} G(0)(\bar{X}) = -\frac{1}{2} :Q^{\gamma b} Q^{\beta b}:.
\end{split} 
 \end{equation*}
It is straightforward to check that $X,\bar{X}, Y, \bar{Y}$ lie in $\cC(\gs\gl_2,\mathbb{C}^3)$ and are primary of weights $\frac{3}{2}, \frac{3}{2}, 2,2$ with respect to $L$.

\begin{lemma} The fields $F, L, G, \bar{G}, X, \bar{X}, Y, \bar{Y}$ satisfy the operator product relations of Odake's algebra $\cO$, as well as the normally ordered polynomial relations \eqref{odakerelations}.
\end{lemma}
\begin{proof} This is a straightforward calculation. \end{proof}

\begin{lemma}\label{reformlem} We have an isomorphism of vertex algebras $$\cC(\gs\gl_2,\mathbb{C}^3) \cong \text{Com}(V_0(\gs\gl_2), V_{-4}(\gs\gl_2) \otimes \cE).$$
\end{lemma}
\begin{proof} Recall that $\Theta_{\cS}$ and $\cS^{\gs\gl_2[t]}$ form a Howe pair inside $\cS$, and that $\Theta_{\cS}$ is isomorphic to $V_{-4}(\gs\gl_2)$. It follows that $$\cC(\gs\gl_2,\mathbb{C}^3) =\text{Com}(V_0(\gs\gl_2) \otimes \cS^{\gs\gl_2[t]}, \cW)  = \text{Com}(V_0(\gs\gl_2) ,\text{Com}(\cS^{\gs\gl_2[t]}, \cW)) $$ $$ =  \text{Com}(V_0(\gs\gl_2) , \Theta_{\cS} \otimes \cE) = \text{Com}(V_0(\gs\gl_2) ,V_{-4} (\gs\gl_2) \otimes \cE).$$
\end{proof}

For the sake of illustration, we rewrite the generators $L, G, \bar{G}, Y, \bar{Y}$ in terms of the generators $X^x, X^y, X^h$ of $V_{-4}(\gs\gl_2)$.

\begin{equation*} 
\begin{split}
L&= - \frac{1}{2} :X^x X^y: - \frac{1}{2} :X^y X^x: - \frac{1}{4} :X^h X^h:  - \frac{1}{2} :X^x \Theta^y_{\cE}:  - 
 \frac{1}{2} :X^y \Theta^x_{\cE}: - \frac{1}{4} :X^h \Theta^h_{\cE}:\\  &\qquad + \frac{1}{2} :FF: + \partial F - 2 L_{\cE},\\
G&=\  :X^x c^{x'}: +:X^y c^{y'}: +:X^h c^{h'}:  + \frac{1}{2}(:\Theta^x_{\cE} c^{x'}: +:\Theta^y_{\cE} c^{y'}: +:\Theta^h_{\cE} c^{h'}: ),\\
  \bar{G}& = -\frac{1}{2} (:X^x b^y: + :X^y b^x: + \frac{1}{2} :X^h b^h: ) -\frac{1}{4}(:\Theta^x_{\cE} b^y: + :\Theta^y_{\cE} b^x: + \frac{1}{2} :\Theta^h_{\cE} b^h: ),\\
   Y& = -\frac{1}{8}(: X^h c^{x'} c^{y'}: - 2:X^x c^{x'} c^{h'}: + 2 :X^y c^{y'} c^{h'}:) - \\
   &\qquad\frac{1}{16} (: \Theta^h_{\cE} c^{x'} c^{y'}: - 2:\Theta^x_{\cE} c^{x'} c^{h'}: + 2 :\Theta^y_{\cE} c^{y'} c^{h'}:),\\
   \bar{Y}& = \frac{1}{2}(: X^h b^{x} b^{y}: + :X^x b^{y} b^{h}: - :X^y b^{x} b^{h}:) + \frac{1}{4} (: \Theta^h_{\cE} b^{x} b^{y}: + :\Theta^x_{\cE} b^{y} b^{h}: - :\Theta^y_{\cE} b^{x} b^{h}:).
   \end{split}
   \end{equation*}

For an arbitrary $k\in \mathbb{C}$, we may consider the commutant $$\cC_k=\text{Com}(V_{k+4}(\gs\gl_2), V_{k}(\gs\gl_2) \otimes \cE),$$ where the generators of $V_k(\gs\gl_2)$ and $V_{k+4}(\gs\gl_2)$ are $X^{\xi}$ and $X^{\xi}+ \Theta^{\xi}_{\cE}$, respectively, for $\xi = x,y,h$. Then $\cC(\gs\gl_2,\mathbb{C}^3) \cong \cC_{-4}$. This deformation of $\cC(\gs\gl_2,\mathbb{C}^3)$ is a special case of a construction that is well known in the physics literature (see \cite{BFH}), and we have $$\lim_{k\ra \infty} \cC_k \cong \cE^{SL_2}.$$ Moreover, there is a linear map $\cC_k \ra \cE^{SL_2}$ defined as follows. Each element $\omega \in \cC_k$  of weight $d$ can be written uniquely in the form $\omega = \sum_{r=0}^d \omega_r$ where $\omega_r$ lies in the space \begin{equation*} \label{shapemon}(V_k(\gs\gl_2) \otimes \cE)^{(r)}\end{equation*} spanned by terms of the form $\alpha \otimes \nu$ where $\alpha \in V_k(\gs\gl_2)$ has weight $r$. Clearly $\omega_0 \in \cE^{SL_2}$ so we have a well-defined linear map \begin{equation*} \label{limitmap} \phi_k: \cC_k \ra \cE^{SL_2},\ \ \ \ \ \ \omega \mapsto \omega_0.\end{equation*} Note that $\phi_k$ is not a vertex algebra homomorphism for any $k$. By the same argument as Lemmas 8.3 and 8.4 of \cite{CL}, $\phi_k$ is a linear isomorphism whenever $V_k(\gs\gl_2)$ is a simple vertex algebra. The values of $k$ for which $V_k(\gs\gl_2)$ is simple were determined by Kac-Gorelik in \cite{GK}, and in particular $V_{-4}(\gs\gl_2)$ is simple. We conclude that $\cC(\gs\gl_2,\mathbb{C}^3)$ has the same graded character as $\cE^{SL_2}$.

We are now in a position to prove our main result.

\begin{thm} \label{mainodake} $\cC(\gs\gl_2,\mathbb{C}^3)$ is strongly generated by the fields $F, L, G, \bar{G}, X, \bar{X}, Y, \bar{Y}$, and is isomorphic to Odake's algebra $\cO$.
\end{thm}

\begin{proof} Since $\cO$ is a simple vertex algebra, any vertex algebra with the same generators, operator product relations, and graded character must be isomorphic to $\cO$. Since $\cC(\gs\gl_2,\mathbb{C}^3)$ contains a subalgebra with the same generators and operator product relations as $\cO$, and $\cC(\gs\gl_2,\mathbb{C}^3)$ has the same graded character as $\cE^{SL_2}$, it suffices to show that the graded characters of $\cE^{SL_2}$ and $\cO$ coincide. First, note that $\cE$ is triply graded by conformal weight and the eigenvalues of $F(0)$ and $\Theta^{h}_{\cE}(0)$, and 

$$\text{ch}[\cE](z;w;q)=\text{tr}_{\cE}(q^{L(0)}z^{F(0)}w^{\Theta^h_{\cE}(0)} ) $$ 
\begin{equation*}
\begin{split}
&=\prod_{n= 1}^{\infty} (1+q^{n- \frac{1}{2}} z w^2) (1+q^{n-\frac{1}{2}} z^{-1} w^2)(1+q^{n-\frac{1}{2}} z w^{-2})(1+q^{n- \frac{1}{2}} z^{-1} w^{-2})(1+q^{n- \frac{1}{2}} z )(1+q^{n- \frac{1}{2}} z^{-1})\\
&=\prod_{n= 1}^{\infty}\frac{(1+zq^{n-\frac{1}{2}})(1+z^{-1}q^{n-\frac{1}{2}})}{(1-q^n)^2}\sum_{m,s\in\mathbb Z}z^{m+s}w^{2(m-s)}q^{m^2/2+s^2/2}.
\end{split}
\end{equation*}
For the last equality, we used Jacobi's triple product formula. 
Then $\text{ch}[\cE^{SL_2}](z;q)$ is obtained from $\text{ch}[\cE](z;w;q)$ by taking the difference between the coefficients of $w^0$ and $w^2$, namely
\begin{equation*}
\begin{split}
\text{ch}[\cE^{SL_2}](z;q)&=
\prod_{n=1}^\infty \frac{(1+zq^{n-\frac{1}{2}})(1+z^{-1}q^{n-\frac{1}{2}})}{(1-q^n)^2}\sum_{m\in\mathbb Z}q^{m^2}z^{2m}-q^{m^2+m+\frac{1}{2}}z^{2m+1}
=\text{ch}[\cO](z;q).
\end{split}
\end{equation*} 
 \end{proof}

Finally, combining Theorems \ref{howepair}, \ref{spiece}, and \ref{mainodake}, we see that $V_{-3/2}(\gs\gl_2)$ and $\cO$ form a Howe pair inside $\cW^{\gs\gl_2[t]}$.

\section{The case where $\gg= \gs\gl_2$ and $V$ is the standard module}
In this section, we consider $\cE^{\gg[t]}$, $\cS^{\gg[t]}$, $\cW^{\gg[t]}$, and $\cC(\gg,V)$ in the simplest case where $\gg = \gs\gl_2$ and $V$ is the standard module $\mathbb{C}^2$. In this case, a description of $\cS^{\gg[t]}$, $\cE^{\gg[t]}$, and $\cW^{\gg[t]}$ appears in \cite{LSS}.

\begin{thm} Let $V$ the standard $\gs\gl_2$-module $\mathbb{C}^2$.  
\begin{enumerate} 
 \item $\cS^{\gs\gl_2[t]}$ is a rank one Heisenberg algebra with generator
 $:\beta^1 \gamma^1: +:\beta^2\gamma^2:$.

\item $\cE^{\gs\gl_2[t]}$ is isomorphic to the irreducible quotient $L_1(\gs\gl_2)$ of the affine vertex algebra $V_1(\gs\gl_2)$, and has 
generators $:b^1 c^1: + :b^2 c^2:, :b^1 b^2:$, and $:c^1 c^2:$.

\item $\cW^{\gs\gl_2[t]}$ is a homomorphic image of
$V_1(\gs\gl(2|1))$ with generators
\begin{equation}\label{sl21}
\begin{split}
H &=\ :\beta^1 \gamma^1: + :\beta^2\gamma^2:, \\
F&=-:b^1 c^1: - :b^2 c^2:,\quad
E^+=\ :b^1 b^2:,\quad
E^-=\ :c^1 c^2:,\\
Q_1^-&=\ :\beta^1 c^1: + :\beta^2 c^2:,\quad
Q_1^+=\ :b^1 \gamma^1: + :b^2\gamma^2:,\\
Q_2^+&=\ :b^1 \beta^2: - :b^2 \beta^1:,\quad
Q_2^-=\ :\gamma^1 c^2: - :\gamma^2 c^1:.
\end{split}
\end{equation}

\end{enumerate}
\end{thm}

We now give an alternative realization of $V_1(\gs\gl(2|1))$. For this, we need the Friedan-Martinec-Shenker bosonization of bosons \cite{FMS}. 
This is an embedding of the rank $1$ $\beta\gamma$-system with generators $\beta,\gamma$ into the vertex algebra $V_L \otimes \cF$, 
where $L$ is the lattice $ \mathbb{Z}$ and $\cF$ is the symplectic fermion vertex algebra with generators $\chi^{\pm}$ and operator product relations $\chi^+(z) \chi^-(w) \sim (z-w)^{-2}$. Let $\psi$ be a boson with operator product expansion
$$\psi(z)\psi(w) \sim - \ln(z-w),$$ and let $\eta= 1$ be the generator of $L$, so that $V_L$ has generators $\Gamma_{\pm \eta}$. The $\beta\gamma$-system embeds into $V_L \otimes \cF$ via the map $$\beta\mapsto \ : \Gamma_{\eta} \chi^+:,\qquad \gamma \mapsto \ :\Gamma_{-\eta} \chi^-:.$$

Now we bosonize the entire $bc\beta\gamma$-system $\cW$. For $j=1,2$, denote by $\phi_{j},\psi_j$ bosons with operator product expansions
$$\phi_j (z)\phi_j (w) \sim \ln(z-w),\qquad \psi_j (z)\psi_j(w) \sim -\ln(z-w),$$ respectively. Let $L$ be the rank $4$ lattice $\mathbb{Z}^4$ with generators 
\begin{equation*} \xi_1 = (1,0,0,0),\qquad \xi_2 = (0,1,0,0),\qquad \eta_1 = (0,0,1,0),\qquad \eta_2 = (0,0,0,1),\end{equation*}
and quadratic form
\begin{equation*}
 \xi_j\xi_j=1,\qquad \eta_j\eta_j=-1,
\end{equation*}
so that the lattice has signature $(2,2)$.
For $j=1,2$ let $\Gamma_{\pm \xi_j}, \Gamma_{\pm \eta_j}$ be the corresponding generators for $V_L$. Let $\cF$ be the rank $2$ symplectic fermion algebra with generators $\chi^{\pm}_j$ and operator products $\chi^+_j(z) \chi^-_j(w) \sim (z-w)^{-2}$. We have an embedding $\cW \ra V_L \otimes \cF$ given by

\begin{equation}\label{bos}
\begin{split}
b^j\mapsto \Gamma_{\xi_j} ,\qquad c^j \mapsto \Gamma_{-\xi_j},\qquad :b^jc^j:\ \mapsto \partial \phi_j \\
\beta^j \mapsto\ : \Gamma_{\eta_j} \chi^+_j:,\qquad \gamma^j \mapsto\ :\Gamma_{-\eta_j} \chi^-_j:,\qquad :\beta^j\gamma^j: \ \mapsto  \partial \psi_j.
\end{split}
\end{equation}

Define $\phi_\pm=\phi_1\pm \phi_2$ and $\psi_\pm=\psi_1\pm \psi_2$. We have the corresponding elements $\xi_{\pm} = \xi_1\pm \xi_2$ and $\eta_{\pm} = \eta_1\pm \eta_2$ of the lattice $L$. Define \begin{equation}\label{fer}
\begin{split}
\sqrt{2}\chi^+&= \ :\Gamma_{\frac{1}{2}(\zeta_- -\eta_-)} \chi^+_2:  + :\Gamma_{\frac{1}{2}(-\zeta_- +\eta_-)} \chi^+_1:,\\
\sqrt{2}\chi^-&= \ :\Gamma_{\frac{1}{2}(-\zeta_-+\eta_-)} \chi^-_2: + \ :\Gamma_{\frac{1}{2}(\zeta_- -\eta_-)} \chi^-_1:.
\end{split}
\end{equation}

The operators products of $\chi^\pm, \phi_+$ and $\psi_+$ are
\begin{equation} \label{newop}
 \phi_+(z)\phi_+(w) \sim 2\ln(z-w),\quad
 \psi_+(z)\psi_+(w) \sim-2\ln(z-w),\quad
\chi^+(z)\chi^-(w) \sim (z-w)^{-2}. 
\end{equation}

\begin{lemma}
The generators of $\cW^{\gs\gl_2[t]}$ can be expressed in the form
\begin{equation*}\begin{split}
H &=  \partial \psi_+, \\
F&= - \partial \phi_+,\quad
E^+= \Gamma_{\xi_+},\quad
E^-= \Gamma_{-\xi_+},\\
Q_1^-&= \sqrt{2} :\Gamma_{\frac{1}{2}(-\xi_+ +\eta_+)}\chi^+:,\quad
Q_1^+= \sqrt{2}:\Gamma_{\frac{1}{2}(\xi_+-\eta_+)}\chi^-:,\\
Q_2^+&= \sqrt{2}:\Gamma_{\frac{1}{2}(\xi_+ +\eta_+)}\chi^+:,\quad
Q_2^-=\sqrt{2}: \Gamma_{\frac{1}{2}(-\xi_+-\eta_+)}\chi^-:.
\end{split}\end{equation*} 
\end{lemma}

\begin{proof}
This follows by inserting \eqref{bos} and \eqref{fer} into \eqref{sl21}.
\end{proof}

\begin{thm}
$\cC(\gs\gl_2,\mathbb{C}^2)\cong L_1(\gs\gl_2)\otimes \cW_{3,-2}$, where $\cW_{3,-2}$ denotes the Zamolodchikov $\cW_3$-algebra with central charge $c=-2$.
\end{thm}
\begin{proof}

We denote the lattice vertex algebra whose generators are $\Gamma_{\pm\frac{1}{2}\eta_+}$ by $V_{\psi}$. Let
$V_{\phi}$ be the lattice vertex algebra generated by $\Gamma_{\pm\frac{1}{2}\xi_+}$, and call $V_{\text{SF}}$ the vertex algebra generated by $\chi^\pm$, which is a symplectic fermion algebra by \eqref{newop}. Introduce a $\mathbb Z$-grading $V_{\text{SF}}=\oplus V_n$ by saying that an element of the form $:\partial^{a_1}\chi^+ \cdots \partial^{a_r}\chi^+\partial^{b_1}\chi^-\cdots \partial^{b_s}\chi^-$ has grade $r-s$.
Clearly $\cW^{\gs\gl_2[t]}$ is a subalgebra of $V_{\psi}\otimes V_{\phi}\otimes V_{\text{SF}}$, and $\text{Com}(\cS^{\gs\gl_2[t]},V_{\psi}\otimes V_{\phi}\otimes V_{\text{SF}})=V_{\phi}\otimes V_{\text{SF}}$. Hence
$$\text{Com}(\cS^{\gs\gl_2[t]},\cW^{\gs\gl_2[t]})= \cW^{\gs\gl_2[t]}\cap  (V_{\phi}\otimes V_{\text{SF}})   =     L_1(\gs\gl_2)\otimes ( \cW^{\gs\gl_2[t]}\cap  V_{\text{SF}}).$$
$\cW^{\gs\gl_2[t]}$ is graded by both $H(0)$ and the $\mathbb Z$-grading of symplectic fermions. The two gradings differ by a factor of $-2$, hence 
$\cW^{\gs\gl_2[t]}\cap  V_{\text{SF}}\subset V_0$.
We will show that all fields of the form $:\partial^a \chi^+\partial^b \chi^-:$ are in $\cW^{\gs\gl_2[t]}$, which implies that \begin{equation} \label{intermediate} \cC(\gs\gl_2,\mathbb{C}^2)\cong L_1(\gs\gl_2)\otimes V_0.\end{equation}
For this, we introduce an ordering as follows $:\partial^a \chi^+\partial^b \chi^-:\  >\ :\partial^{a'}\chi^+\partial^{b'}\chi^-:$
if $a+b>a'+b'$ or if $a+b=a'+b'$ and $a>a'$. 
We prove the claim by induction on this ordering.
Certainly $1\in V_0\cap\cW^{\gs\gl_2[t]}$.
We compute
\begin{equation*}
:\partial^a Q_1^- \partial^b Q^+_1:\ = \ - 2 :\partial^a \chi^+\partial^b \chi^-: + V(\chi^\pm,F+H)
\end{equation*}
where  $V(\chi^\pm,F+H)$ is a field expressed in terms of $F+H$ and its derivatives as well as fields of lower degree in $V_0$. By the induction hypothesis, $V(\chi^\pm,F+H)$ in $\cW^{\gs\gl_2[t]}$, and hence the same is true for all fields of the form $:\partial^a \chi^+\partial^b \chi^-:$. Furthermore, every field of $V_0$ can be written as a normally ordered polynomial in such fields. This completes the proof of \eqref{intermediate}.

Finally, note that $V_0$ is generated as a vertex algebra by the fields $$L =\ :\chi^+ \chi^-:,\ \ \ \ \ \ \ \ W=  \frac{1}{\sqrt{6}} :(\partial\chi^+) \chi^-: -  \frac{1}{\sqrt{6}} :\chi^+ \partial \chi^-:$$ of weights $2$ and $3$. It is not difficult to check that $W$ is primary of weight $3$ with respect to $L$, and satisfies the operator product
$$W(z)W(w)\sim -\frac{2}{3} (z-w)^{-6} + 2L(w) (z-w)^{-4} + \partial L(w)(z-w)^{-3}$$ $$+ (\frac{8}{3} :L(w)L(w): - \frac{1}{2} \partial^2 L(w))(z-w)^{-2} $$ $$+(\frac{4}{3} \partial(:L(w)L(w):) - \frac{1}{3} \partial^3 L(w))(z-w)^{-1}.$$ This shows that $V_0$ coincides with the Zamolodchikov $\cW_3$-algebra with $c=-2$ \cite{Za}. \end{proof}

\begin{remark}
A family of quasi-rational vertex algebras \cite{AMI,AMII} are the $\cW(p,q)$ triplet theories for positive co-prime integers $p<q$, 
which were first introduced in physics in \cite{Kau}. 
These have central charge $c=1-6(p-q)^2/(pq)$, and the algebra is strongly generated by the Virasoro field and three fields of
conformal dimension $(2p-1)(2q-1)$. The algebra contains a singlet subalgebra with only one generator in addition to the Virasoro field.
The simplest case is $\cW(1,2)$, which is known to be a subalgebra of symplectic fermions \cite{GKau}. 
$V_0$ also coincides with the singlet subalgebra of the $\cW(1,2)$ triplet vertex algebra.
\end{remark}
\begin{remark}
We have realized the singlet algebra at $c=-2$ as a commutant involving $V_1(\gs\gl(2|1))$ and also $V_{-1/2}(\gs\gl_2)$.
There are related observations, namely there are constructions \cite{R,CR} relating the $\cW(1,2)$ singlet and triplet algebra to $V(\gg\gl(1|1))$ and $V_{-1/2}(\gs\gl_2)$
as well as the $\cW(1,3)$ singlet algebra to $V_{-4/3}(\gs\gl_2)$ \cite{A}. Moreover $V_1(\gs\gl(2|1))$ is a simple current extension
of $V(\gg\gl(1|1))$ \cite{CR}.  
\end{remark}

\end{document}